\documentclass[11pt]{amsart}
\usepackage[margin=1in]{geometry}
\usepackage{amsmath,amssymb,mathtools}
\usepackage[hidelinks]{hyperref}
\hypersetup{pdftitle={A reciprocal Schoenberg inequality and its extensions},
 pdfsubject={Reciprocal moments of critical points and positive-charge equilibrium points},
 pdfkeywords={Schoenberg inequality, critical points, exterior powers, normal matrices}}
\numberwithin{equation}{section}
\newtheorem{theorem}{Theorem}[section]
\newtheorem{corollary}[theorem]{Corollary}
\newtheorem{lemma}[theorem]{Lemma}
\newtheorem{proposition}[theorem]{Proposition}
\theoremstyle{remark}
\newtheorem{remark}[theorem]{Remark}
\newtheorem{example}[theorem]{Example}
\DeclareMathOperator{\tr}{tr}
\DeclareMathOperator{\diag}{diag}
\DeclareMathOperator{\Rea}{Re}
\DeclareMathOperator{\Ima}{Im}

\newcommand{\R}{\mathbb R}

\newcommand{\abs}[1]{\left|#1\right|}
\newcommand{\E}{\mathcal E}
\newcommand{\B}{\mathcal B}
\title[Reciprocal Schoenberg inequalities]{A reciprocal Schoenberg inequality and its extensions}
\author{Olga Holtz}
\address{University of California, Berkeley}
\date{}
\subjclass[2020]{Primary 30C10; Secondary 15A18, 15A42, 30C15}
\keywords{Zeros of polynomials, critical points, Schoenberg inequality, majorization, exterior powers, normal matrices}
\begin{document}
\begin{abstract}
We study reciprocal moments of the critical points of a polynomial with a simple zero at the origin. A positive similarity transforms a reciprocal derivative companion into a complex symmetric matrix. This gives a short proof of the reciprocal Schoenberg inequality and a refinement for every elementary symmetric function. We determine the optimal constants in the corresponding inequalities for every positive exponent, including the sublinear range, and extend the results to arbitrary positive charges. We also characterize the normal powers of the companion and of its symmetric representative. The resulting criterion gives the complete equality cases in two quartic inequalities and explains the appearance of noncollinear equality configurations at higher orders. We also obtain a closed generating determinant and a quantitative collinearity estimate.
\end{abstract}
\maketitle

\section{Introduction}

The Gauss--Lucas theorem places the critical points of a polynomial in the convex hull of its zeros. Moment inequalities give a different kind of information: they compare the distribution of the critical points with that of the zeros. A distinguished example is Schoenberg's inequality. If a polynomial of degree $N\ge2$ has zeros $\zeta_1,\ldots,\zeta_N$ and critical points $\omega_1,\ldots,\omega_{N-1}$, then
\begin{equation}\label{eq:classical}
 \sum_{j=1}^{N-1}|\omega_j|^2
 \le \frac{N-2}{N}\sum_{j=1}^N|\zeta_j|^2
       +\frac1{N^2}\abs{\sum_{j=1}^N\zeta_j}^2.
\end{equation}
Equality holds precisely when the zeros are collinear. Schoenberg proposed the centered form of this inequality in 1986 \cite{Schoenberg}. Pereira \cite{Pereira} and Malamud \cite{Malamud} proved it through matrix representations of differentiation. These representations also connect polynomial geometry with the majorization relations of de Bruijn and Springer.

Cheung and Ng \cite{CN} developed derivative companion matrices and used them to prove higher-order inequalities, including a conjecture of de Bruin and Sharma. Kushel and Tyaglov \cite{KT} obtained a sharper quartic inequality through circulant matrices. Borcea \cite{Borcea} extended the majorization framework to equilibrium points of arbitrary positive charge distributions. Exterior powers already play an essential role in this literature: they replace individual eigenvalues by products of eigenvalues.

Negative moments arise naturally when one measures distances from a distinguished zero. We therefore write
\begin{equation}\label{eq:P}
 P(z)=z\prod_{j=1}^n(z-z_j),\qquad z_j\ne0,
\end{equation}
and denote the zeros of $P'$ by $w_1,\ldots,w_n$. Since $P'(0)\ne0$, all the $w_j$ are nonzero. Tang and Zhang \cite[Conjecture 6.3]{TZ} proposed the reciprocal inequality
\begin{equation}\label{eq:reciprocal}
 \sum_{j=1}^n|w_j|^{-2}
 \le 3\sum_{j=1}^n|z_j|^{-2}
       +\abs{\sum_{j=1}^nz_j^{-1}}^2.
\end{equation}

Their Conjecture 1.13 asks for the optimal power-sum bound
\begin{equation}\label{eq:TZpower}
 \sum_{j=1}^n|w_j|^{-p}
 \le \frac{(n+1)^p+n-1}{n}\sum_{j=1}^n|z_j|^{-p},
 \qquad p\ge1.
\end{equation}
We prove both inequalities, give their equality cases, and place them in a larger family. At $p=1$, the equality assertion in Conjecture 1.13 requires a correction: the nonzero zeros need only lie on one ray, rather than coincide.

The basic observation is simple. Put $D=\diag(z_1^{-1},\ldots,z_n^{-1})$ and let $J$ be the all-ones matrix. The matrix $Q=D(I+J)$ has eigenvalues $w_j^{-1}$. Its similar representative
\[
 C=(I+J)^{1/2}D(I+J)^{1/2}
\]
is complex symmetric and has a smaller Frobenius norm. Schur's inequality applied to $C$ proves \eqref{eq:reciprocal}; applying it to exterior powers of $C$ retains the same collinearity equality condition. A second factorization removes the phases of the $z_j^{-1}$ and leads, through the Araki--Lieb--Thirring inequality, to sharp constants for all positive exponents.

For unit charges, the phase comparison itself is contained in Cheung--Ng's Theorem 2.1 after inversion and the use of the product identity. Likewise, the endpoint $p=1$ of our elementary symmetric bounds follows from Borcea's product majorization. Our bounds for $p>1$ improve the coefficients supplied by that majorization, as explained in Remark~\ref{rem:comparison}. The additional phase-sensitive refinements and the equality analysis use the same companion, but require its symmetric representative. In particular, a divided-difference formula characterizes all its normal powers. This gives the complete equality classification left open in \cite[Remark 7.2]{TZ} and shows why the refined quartic inequality has only collinear equality cases, whereas higher even orders need not.

\section{Definitions and notation}

Throughout, $z_1,\ldots,z_n$ are nonzero complex numbers and $r_1,\ldots,r_n$ are positive real numbers. Repetitions among the $z_j$ are allowed. Define
\begin{equation}\label{eq:Fr}
 F_r(z)=\prod_{j=1}^n(z-z_j)
       +z\sum_{j=1}^nr_j\prod_{\ell\ne j}(z-z_\ell).
\end{equation}
Let $w_1,\ldots,w_n$ be its zeros, counted with multiplicity, and put
\[
 b_j=z_j^{-1},\qquad \lambda_j=w_j^{-1},\qquad
 \sigma=\sum_jr_j,\qquad u=(\sqrt{r_1},\ldots,\sqrt{r_n})^T.
\]
The polynomial $F_r$ has degree $n$, leading coefficient $1+\sigma$, and a nonzero constant term. Thus all these reciprocals are defined.

For distinct positions, the $w_j$ are the zeros of
\[
 \frac1z+\sum_j\frac{r_j}{z-z_j},
\]
the force function associated with positive charges $1,r_1,\ldots,r_n$. For repeated positions we use the polynomial definition \eqref{eq:Fr}, so that roots lost through rational cancellation are retained. Arbitrary positive charges $\alpha_0,\ldots,\alpha_n$ are obtained by taking $r_j=\alpha_j/\alpha_0$. The choice $r_j=1$ gives $F_r=P'$ for \eqref{eq:P}. More generally, $r_j=1/\ell$ gives the $n$ nonzero critical points of $z^\ell\prod_j(z-z_j)$, where $\ell$ is a positive integer.

Our matrices are
\begin{equation}\label{eq:matrices}
 D=\diag(b_1,\ldots,b_n),\quad K=I+uu^T,\quad
 H=K^{1/2},\quad Q=DK,\quad C=HDH.
\end{equation}
Here $H$ is the positive square root. The matrices $K,H$ are real symmetric, $C^T=C$, and $C=HQH^{-1}$. Lemma~\ref{lem:representation} identifies their spectrum with $(\lambda_j)$.

For a vector $x$, let $e_k(x)$ denote its $k$th elementary symmetric function. We abbreviate
\[
 \E_{k,p}(w)=e_k(|w_1|^{-p},\ldots,|w_n|^{-p}),\qquad
 b_I=\prod_{j\in I}b_j,\qquad r_I=\sum_{j\in I}r_j,
\]
with $b_\varnothing=1$. A sum over $|I|=k$ runs over all $k$-element subsets of $\{1,\ldots,n\}$. The Frobenius and operator norms are denoted by $\|\cdot\|_F$ and $\|\cdot\|_{\rm op}$. We use $X^*$ for the conjugate transpose, $\bar X$ for entrywise conjugation, and $[X,Y]=XY-YX$. Collinearity below always means collinearity on a line through the origin. It is equivalent for the configurations $(z_j)$ and $(b_j)$.

\section{Main results}

\subsection{Sharp inequalities for all positive exponents}

\begin{theorem}\label{thm:allpowers}
For $1\le k\le n$ and $p\ge1$,
\begin{equation}\label{eq:weightedpowers}
 \E_{k,p}(w)\le
 \sum_{|I|=k}\left(1+\frac{(1+\sigma)^p-1}{\sigma}r_I\right)|b_I|^p.
\end{equation}
For $0<p<1$,
\begin{equation}\label{eq:sublinear}
 \E_{k,p}(w)\le\sum_{|I|=k}(1+r_I)^p|b_I|^p.
\end{equation}
If $k<n$, equality in \eqref{eq:weightedpowers} holds exactly when all $b_j$ coincide for $p>1$, and exactly when they have the same argument for $p=1$. Inequality \eqref{eq:sublinear} is strict for $k<n$, but each of its coefficients is individually optimal. For $k=n$, both inequalities are identities for every $p>0$.
\end{theorem}

\begin{corollary}\label{cor:constants}
Suppose $r_1=\cdots=r_n=\rho>0$. The best constant in
\[
 \E_{k,p}(w)\le A_{n,k,p}(\rho)\,e_k(|z_1|^{-p},\ldots,|z_n|^{-p})
\]
is
\begin{equation}\label{eq:constants}
 A_{n,k,p}(\rho)=
 \begin{cases}
 (1+k\rho)^p,&0<p\le1,\\[1mm]
 1+\dfrac{k}{n}\bigl((1+n\rho)^p-1\bigr),&p\ge1.
 \end{cases}
\end{equation}
For $k<n$, the equality cases are those in Theorem~\ref{thm:allpowers}; in the sublinear range the optimal constant is not attained. For $k=n$, equality always holds.
\end{corollary}

Taking $k=1$ and $\rho=1$ proves \eqref{eq:TZpower}. For example, $P(z)=z(z-1)(z-2)$ attains equality at $p=1$, although its nonzero zeros are distinct. At the other end, $k=n$ records the product identity
\begin{equation}\label{eq:product}
 \prod_j\lambda_j=(1+\sigma)\prod_jb_j.
\end{equation}

\subsection{A refinement retaining the phases}

\begin{theorem}\label{thm:exterior}
For $1\le k\le n$, define
\begin{align}\label{eq:Bk}
 \B_k(b;r)={}&
 \sum_{|I|=k}\left((1+r_I)^2-\sum_{j\in I}r_j^2\right)|b_I|^2\notag\\
 &+\sum_{|A|=k-1}|b_A|^2\abs{\sum_{j\notin A}r_jb_j}^2.
\end{align}
Then
\begin{equation}\label{eq:phasebound}
 \E_{k,2}(w)\le\B_k(b;r).
\end{equation}
For $k<n$, equality holds if and only if the $z_j$ are collinear. For $k=n$, equality always holds. Moreover,
\begin{align}\label{eq:exteriordeficit}
 &\sum_{|I|=k}\bigl(1+(\sigma+2)r_I\bigr)|b_I|^2-\B_k(b;r)\notag\\
 &\hspace{12mm}=\sum_{|A|=k-1}|b_A|^2
       \sum_{\substack{i<j\\i,j\notin A}}r_ir_j|b_i-b_j|^2.
\end{align}
\end{theorem}

For unit charges this reads
\begin{equation}\label{eq:unitexterior}
 \E_{k,2}(w)\le(k^2+k+1)e_k(|b|^2)
       +\sum_{|A|=k-1}|b_A|^2\abs{\sum_{j\notin A}b_j}^2.
\end{equation}
At $k=1$ it is precisely \eqref{eq:reciprocal}, including its complete equality characterization. Identity \eqref{eq:exteriordeficit} measures the improvement over Theorem~\ref{thm:allpowers} at $p=2$ by a sum of squared pairwise distances.

\subsection{Normal powers and even moments}

Define monic polynomials $f_m$ of degree $m$ by
\begin{equation}\label{eq:recurrence}
 f_1(t)=t,\qquad a_m=\sum_jr_jf_m(b_j),\qquad
 f_{m+1}(t)=t\bigl(f_m(t)+a_m\bigr).
\end{equation}
Thus $f_2(t)=t^2+st$, where $s=\sum_jr_jb_j$.

\begin{theorem}\label{thm:normalpowers}
For every positive integer $m$,
\begin{equation}\label{eq:evensandwich}
 \sum_j|w_j|^{-2m}\le\|C^m\|_F^2\le\|Q^m\|_F^2,
\end{equation}
and the second difference is the weighted variance
\begin{equation}\label{eq:variance}
 \|Q^m\|_F^2-\|C^m\|_F^2
 =\sigma\sum_jr_j|f_m(b_j)|^2-|a_m|^2.
\end{equation}
The following equality criteria hold:
\begin{enumerate}
\item $Q^m$ is normal if and only if $f_m(b_1)=\cdots=f_m(b_n)$.
\item $C^m$ is normal if and only if the $b_j$ are collinear or $f_m(b_1)=\cdots=f_m(b_n)$.
\end{enumerate}
Thus equality in the first inequality of \eqref{eq:evensandwich} has exactly the configurations in (2), while equality between its first and last terms has exactly those in (1). Every noncollinear equality configuration has at most $m$ distinct nonzero zeros.
\end{theorem}

For the quartic case, introduce
\begin{equation}\label{eq:moments}
 \begin{gathered}
 S=\sum_jr_j|b_j|^2,\qquad s=\sum_jr_jb_j,\qquad T=\sum_jr_jb_j^2,\\
 V_0=\sum_j|b_j|^4,\qquad V_1=\sum_jr_j|b_j|^4,\qquad
 W=\sum_jr_j|b_j^2+sb_j|^2.
 \end{gathered}
\end{equation}

\begin{corollary}\label{cor:quartic}
One has
\begin{align}\label{eq:quartics}
 \sum_j|w_j|^{-4}
 &\le V_0+2V_1+S^2+2W+|T+s^2|^2\notag\\
 &\le V_0+2V_1+S^2+(\sigma+2)W.
\end{align}
Equality in the first inequality holds if and only if the $z_j$ are collinear. Equality between the first and last expressions holds if and only if either all $b_j$ coincide, or they take exactly two distinct values $a,c$ satisfying
\begin{equation}\label{eq:twovalues}
 (1+R_a)a+(1+R_c)c=0,\qquad
 R_a=\sum_{b_j=a}r_j,\quad R_c=\sum_{b_j=c}r_j.
\end{equation}
\end{corollary}

For $r_j=1$, the last expression in \eqref{eq:quartics} is the bound in \cite[Theorem 7.1]{TZ}. Its equality cases, apart from coincident roots, are therefore
\[
 P(z)=z(z-\zeta)^h(z-\eta)^{n-h},\qquad
 (n-h+1)\zeta+(h+1)\eta=0,
\]
where $1\le h<n$ and $\zeta,\eta$ are nonzero. The opposite pair for $n=2$ was already noted in \cite[Remark 7.2]{TZ}; the displayed relation gives the complete classification.

\section{The symmetric companion and the sharp inequalities}

\subsection{Representation and the quadratic inequality}

\begin{lemma}\label{lem:representation}
The matrices $Q$ and $C$ in \eqref{eq:matrices} have eigenvalues $\lambda_1,\ldots,\lambda_n$. The matrix $C$ is normal if and only if the $b_j$ are collinear, and
\begin{equation}\label{eq:CF}
 \|C\|_F^2=\sum_j(1+2r_j)|b_j|^2+\abs{\sum_jr_jb_j}^2.
\end{equation}
\end{lemma}

\begin{proof}
The rank-one determinant formula gives the polynomial identity
\begin{align*}
 \det(\lambda I-DK)
 &=\prod_j(\lambda-b_j)-\sum_jr_jb_j\prod_{\ell\ne j}(\lambda-b_\ell)\\
 &=\frac{\lambda^n F_r(1/\lambda)}{F_r(0)}.
\end{align*}
The expression on the right is monic and has precisely the required roots. This also proves \eqref{eq:product}. Since $C=HQH^{-1}$, it has the same spectrum.

Set $v=Du$. Direct multiplication gives
\begin{equation}\label{eq:commutator}
 [C,C^*]=H(vv^*-\bar v v^T)H.
\end{equation}
The $(i,j)$ entry of the middle matrix is
$\sqrt{r_ir_j}(b_i\bar b_j-\bar b_i b_j)$.
It vanishes for every pair exactly when the $b_j$ lie on one line through the origin. Finally,
$\|C\|_F^2=\tr(D^*KDK)$; expanding $K=I+uu^T$ gives \eqref{eq:CF}.
\end{proof}

Schur triangularization yields, for any square matrix $X$,
\begin{equation}\label{eq:schur}
 \sum_j|\lambda_j(X)|^2\le\|X\|_F^2,
\end{equation}
with equality exactly when $X$ is normal. Indeed, the difference is the sum of the squared moduli of the strictly upper triangular entries of a Schur form. Lemma~\ref{lem:representation} and \eqref{eq:schur} prove the weighted reciprocal Schoenberg inequality, and hence \eqref{eq:reciprocal}, with equality.

\subsection{Exterior powers}

Write $X^{[k]}=\bigwedge^k X$. In the standard basis indexed by $k$-element subsets, its $(I,J)$ entry is $\det X[I,J]$. The identities
\[
 (XY)^{[k]}=X^{[k]}Y^{[k]},\qquad
 (X^*)^{[k]}=(X^{[k]})^*
\]
follow from Cauchy--Binet. The eigenvalues of $X^{[k]}$ are the $k$-fold products of the eigenvalues of $X$; its singular values are the corresponding products of singular values. For positive definite $X$, functional calculus gives
$(X^a)^{[k]}=(X^{[k]})^a$ for real $a$.

We will use the following elementary equality observation.

\begin{lemma}\label{lem:compoundnormal}
If $X$ is invertible and $1\le k<n$, then $X^{[k]}$ is normal if and only if $X$ is normal.
\end{lemma}

\begin{proof}
It suffices to show that $P\mapsto P^{[k]}$ is injective on positive definite matrices. If $P^{[k]}=R^{[k]}$, then
$(R^{-1/2}PR^{-1/2})^{[k]}=I$.
Every $k$-fold product of the positive eigenvalues of $R^{-1/2}PR^{-1/2}$ is therefore $1$. Comparing products that differ in one index shows that all the eigenvalues are equal; hence each is $1$ and $P=R$. Apply this to $P=XX^*$ and $R=X^*X$.
\end{proof}

\begin{proof}[Proof of Theorem~\ref{thm:exterior}]
Apply \eqref{eq:schur} to $C^{[k]}$. Its squared Frobenius norm is
\begin{equation}\label{eq:compoundtrace}
 \|(C)^{[k]}\|_F^2
 =\tr\bigl((D^{[k]})^*K^{[k]}D^{[k]}K^{[k]}\bigr)
 =\sum_{|I|=|J|=k}\bar b_I b_J\det K[I,J]^2.
\end{equation}
The minors of $K=I+uu^T$ have a particularly simple form:
\[
 \det K[I,I]=1+r_I,\qquad
 \det K[A\cup\{i\},A\cup\{j\}]^2=r_ir_j\quad(i\ne j).
\]
Every other minor is zero, since two or more rows outside the common index set are proportional. Grouping the off-diagonal terms in \eqref{eq:compoundtrace} by $A=I\cap J$ gives exactly \eqref{eq:Bk}. Equality follows from Lemmas~\ref{lem:representation} and \ref{lem:compoundnormal}; the top exterior power is scalar.

For the refinement identity, use the weighted variance formula
\begin{equation}\label{eq:weightedvariance}
 \left(\sum_{j\in J}r_j\right)\sum_{j\in J}r_j|b_j|^2
 -\abs{\sum_{j\in J}r_jb_j}^2
 =\sum_{\substack{i<j\\i,j\in J}}r_ir_j|b_i-b_j|^2
\end{equation}
on each complement $J=A^c$. After multiplication by $|b_A|^2$, the coefficient contributed to $|b_I|^2$ is
$\sum_{i\in I}r_i(\sigma-r_I+r_i)$.
Combining this with the first sum in \eqref{eq:Bk} gives
$1+(\sigma+2)r_I$, which proves \eqref{eq:exteriordeficit}.
\end{proof}

\subsection{Removing the phases}

We recall two standard matrix inequalities in the precise forms needed here. If $X$ is invertible and $s_j(X)$ are its singular values, Weyl's multiplicative inequality gives
\begin{equation}\label{eq:weyl}
 \sum_j|\lambda_j(X)|^p\le\sum_js_j(X)^p,\qquad p>0.
\end{equation}
Equality holds if and only if $X$ is normal. To see the equality assertion, apply strict convexity of $t\mapsto e^{pt}$ to the majorization of logarithms of eigenvalue moduli by logarithms of singular values. Equality identifies the two multisets; \eqref{eq:schur} then implies normality. See \cite{Bhatia} for multiplicative majorization.

For positive definite $X,Y$, the Araki--Lieb--Thirring inequality is
\begin{equation}\label{eq:ALT}
 \tr\bigl((X^{1/2}YX^{1/2})^p\bigr)\le\tr(X^pY^p),\qquad p\ge1.
\end{equation}
For $p>1$, equality holds if and only if $XY=YX$. We use Araki's stronger log-majorization theorem \cite{Araki}; a short justification of the equality condition is useful. For fixed $p>1$, that theorem makes
\[
 g(t)=\tr\bigl((X^{t/2}Y^tX^{t/2})^{p/t}\bigr)
\]
nondecreasing on $t>0$. The function is real analytic there. Equality in \eqref{eq:ALT} says $g(1)=g(p)$, so $g$ is constant first on $[1,p]$ and then on $(0,\infty)$. In particular,
\[
 0=g(p)-g(p/2)=\tfrac12\|[X^{p/2},Y^{p/2}]\|_F^2.
\]
Functional calculus now gives $XY=YX$. This argument uses only the log-majorization theorem; see also Hiai \cite{Hiai} for a careful treatment of equality in related norm inequalities.

\begin{lemma}\label{lem:phases}
Put $R=\diag(|b_1|,\ldots,|b_n|)$ and $A=R^{1/2}KR^{1/2}$. For all $k$ and $p>0$,
\begin{equation}\label{eq:phasecomparison}
 \E_{k,p}(w)\le\tr\bigl((A^{[k]})^p\bigr).
\end{equation}
For $k<n$, equality holds if and only if all $b_j$ have the same argument.
\end{lemma}

\begin{proof}
Write $D=UR$, with $U$ diagonal unitary, and choose a diagonal unitary $V$ with $V^2=U$. The matrix
$M=D^{1/2}KD^{1/2}=VAV$
is similar to $Q$. Its singular values are the eigenvalues of $A$. Apply \eqref{eq:weyl} to $M^{[k]}$.

For $k<n$, equality is equivalent to normality of $M$, by Lemma~\ref{lem:compoundnormal}. Since
$MM^*=VA^2V^*$ and $M^*M=V^*A^2V$,
this is equivalent to $UA^2=A^2U$, and hence to $UA=AU$. Every off-diagonal entry of $A$ is positive. Therefore $U$ is scalar, as required.
\end{proof}

\begin{proof}[Proof of Theorem~\ref{thm:allpowers} and Corollary~\ref{cor:constants}]
For $p\ge1$, apply \eqref{eq:ALT} to $R^{[k]}$ and $K^{[k]}$ in \eqref{eq:phasecomparison}. Since
\[
 K^p=I+\frac{(1+\sigma)^p-1}{\sigma}uu^T,
\]
we obtain
\[
 \E_{k,p}(w)\le
 \tr\bigl((R^p)^{[k]}(K^p)^{[k]}\bigr)
 =\sum_{|I|=k}|b_I|^p
       \left(1+\frac{(1+\sigma)^p-1}{\sigma}r_I\right).
\]
For $k<n$ and $p>1$, equality also requires $R^{[k]}K^{[k]}=K^{[k]}R^{[k]}$. Neighboring index sets $A\cup\{i\}$ and $A\cup\{j\}$ have a nonzero connecting entry in $K^{[k]}$, so $|b_i|=|b_j|$ for every pair. Lemma~\ref{lem:phases} supplies the common argument. At $p=1$, only the latter condition remains.

For $0<p<1$, the diagonal of a positive matrix is majorized by its eigenvalues. Concavity therefore gives
\[
 \tr\bigl((A^{[k]})^p\bigr)
 \le\sum_{|I|=k}\bigl(\det A[I,I]\bigr)^p
 =\sum_{|I|=k}(1+r_I)^p|b_I|^p.
\]
For completeness, diagonalize $A^{[k]}$ unitarily and apply scalar Jensen to each diagonal entry. Strict concavity shows that equality forces $A^{[k]}$ to be diagonal. When $k<n$, its entries joining neighboring index sets are nonzero, so the inequality is strict.

To prove that the coefficient of $|b_I|^p$ cannot be reduced, take $b_j=1$ for $j\in I$ and $b_j=\varepsilon>0$ otherwise, and let $\varepsilon\downarrow0$. The nonzero eigenvalues in the limit are $1+r_I$ and $k-1$ copies of $1$. Thus the left-hand side tends to $(1+r_I)^p$, while every other monomial on the right tends to zero.

For equal charges, the coefficients depend only on $k$, giving \eqref{eq:constants}. The preceding limiting configuration proves sharpness for $p<1$. For $p\ge1$, take all $b_j$ equal: the eigenvalues are $(1+n\rho)b_1$ and $n-1$ copies of $b_1$, and
\[
 \frac{\binom{n-1}{k}+\binom{n-1}{k-1}(1+n\rho)^p}{\binom nk}
 =1+\frac{k}{n}\bigl((1+n\rho)^p-1\bigr).
\]
The case $k=n$ is \eqref{eq:product}.
\end{proof}

\begin{remark}[Comparison with product majorization]\label{rem:comparison}
Borcea's Theorem 3 \cite{Borcea}, applied to products of $n-k$ equilibrium points and the convex function $|\cdot|^p$, yields, after inversion using \eqref{eq:product},
\begin{equation}\label{eq:borcea}
 \E_{k,p}(w)\le(1+\sigma)^{p-1}
       \sum_{|I|=k}(1+r_I)|b_I|^p,\qquad p\ge1.
\end{equation}
Indeed, normalize the charges to total mass one. The coefficient of a product indexed by $J\subseteq\{1,\ldots,n\}$ in the forward inequality is
$1-\sum_{j\in J}r_j/(1+\sigma)$; products containing the distinguished zero vanish. Set $J=I^c$ and divide by the total product of the equilibrium points. This gives \eqref{eq:borcea}. For $k=n$ the product identity suffices.

At $p=1$, \eqref{eq:borcea} and \eqref{eq:weightedpowers} coincide. For $p>1$ and $k<n$, the coefficient in \eqref{eq:borcea} exceeds ours by
\[
 \bigl((1+\sigma)^{p-1}-1\bigr)
       \left(1-\frac{r_I}{\sigma}\right)>0.
\]
Thus the endpoint is already a consequence of product majorization, while the new coefficients give a strict improvement away from that endpoint.
\end{remark}

\section{Normal powers and equality in higher moments}

\subsection{A variance identity}

The next identity explains why the positive similarity improves every even-moment bound, not just the quadratic one.

\begin{lemma}\label{lem:symmetricnorm}
If $L^T=L$, then
\begin{align}
 \|HLH\|_F^2
 &=\|L\|_F^2+2\|Lu\|^2+|u^TLu|^2,\label{eq:HLH}\\
 \|LK\|_F^2
 &=\|L\|_F^2+(\sigma+2)\|Lu\|^2.\label{eq:LK}
\end{align}
\end{lemma}

\begin{proof}
Expand $\tr(L^*KLK)$ using $K=I+uu^T$ to obtain \eqref{eq:HLH}. The two single-rank terms have the same trace because $L$ is symmetric and $u$ is real. For \eqref{eq:LK}, use
$K^2=I+(\sigma+2)uu^T$ in $\tr(L^*LK^2)$.
\end{proof}

Put
\[
 L_m=D(KD)^{m-1}.
\]
This matrix is symmetric, $C^m=HL_mH$, and $Q^m=L_mK$. Moreover,
\begin{equation}\label{eq:Lmu}
 L_mu=f_m(D)u.
\end{equation}
For $m=1$ this is immediate. The induction step follows from
$L_{m+1}=DKL_m$:
\[
 L_{m+1}u=D\bigl(f_m(D)u+a_mu\bigr)=f_{m+1}(D)u.
\]
Lemma~\ref{lem:symmetricnorm} therefore proves \eqref{eq:variance}. Its right-hand side is nonnegative by Cauchy--Schwarz, with equality precisely when all $f_m(b_j)$ coincide. Together with Schur's inequality, this proves the sandwich \eqref{eq:evensandwich}.

Suppose now that $f_m(b_j)=a$ for every $j$. Then $L_mu=au$ and, by symmetry, $u^TL_m=au^T$, so $L_m$ commutes with $K$. Consequently
\[
 Q^m=L_mK=HL_mH=C^m
\]
commutes with both $K$ and $Q$, hence with $D=QK^{-1}$. Since $D^*$ is a polynomial in $D$ (interpolate conjugation on its finite set of distinct diagonal entries), $Q^m$ commutes with $D^*$ and therefore with $Q^*=KD^*$. In particular, $Q^m$ is normal. Conversely, normality of $Q^m$ gives equality between the endpoints of \eqref{eq:evensandwich}, so \eqref{eq:variance} vanishes. This proves part (1) of Theorem~\ref{thm:normalpowers}.

\subsection{Divided differences}

For a polynomial $f$, write
\[
 f[x,y]=
 \begin{cases}
 \dfrac{f(x)-f(y)}{x-y},&x\ne y,\\
 f'(x),&x=y.
 \end{cases}
\]

\begin{lemma}\label{lem:divided}
For $m\ge1$,
\begin{equation}\label{eq:entries}
 (Q^m)_{ij}=\delta_{ij}b_i^m+b_iu_iu_j f_m[b_i,b_j].
\end{equation}
Furthermore,
\begin{equation}\label{eq:powercommutator}
 [C^m,C^*]=H\Omega_mH,\qquad
 (\Omega_m)_{ij}
 =u_iu_j(b_i\bar b_j-\bar b_i b_j)f_m[b_i,b_j].
\end{equation}
\end{lemma}

\begin{proof}
All generating series below may be read formally. Set
\[
 \Psi(t)=1-\sum_j\frac{r_jb_jt}{1-b_jt}.
\]
The recurrence \eqref{eq:recurrence} gives
\begin{equation}\label{eq:fseries}
 \sum_{m\ge1}f_m(x)t^{m-1}
 =\frac{x}{(1-xt)\Psi(t)}.
\end{equation}
Indeed, summing the recurrence first gives
$x(1+t\sum_{m\ge1}a_mt^{m-1})/(1-xt)$; summing this at $x=b_j$ with weights $r_j$ identifies the parenthesis as $1/\Psi(t)$.

The rank-one inverse formula gives
\[
 ((I-tQ)^{-1})_{ij}
 =\frac{\delta_{ij}}{1-b_it}
   +\frac{t b_i u_i u_j}
   {(1-b_it)(1-b_jt)\Psi(t)}.
\]
Taking the divided difference in $x$ of \eqref{eq:fseries} and comparing coefficients proves \eqref{eq:entries}, also when $b_i=b_j$.

Finally,
\[
 H^{-1}[C^m,C^*]H^{-1}=L_mKD^*-D^*KL_m.
\]
Use $L_mK=Q^m$, $KL_m=(Q^m)^T$, and \eqref{eq:entries}. The diagonal terms cancel, leaving \eqref{eq:powercommutator}.
\end{proof}

\begin{proof}[Completion of Theorem~\ref{thm:normalpowers}]
For any matrix $X$, normality of $X^m$ is equivalent to
\begin{equation}\label{eq:finitefuglede}
 [X^m,X^*]=0.
\end{equation}
One implication follows by commuting with $(X^*)^m$. For the other, if $X^m$ is normal, its adjoint is a polynomial in $X^m$, by the spectral theorem and interpolation on the finite spectrum. Since $X$ commutes with $X^m$, it commutes with $(X^m)^*$; taking adjoints gives \eqref{eq:finitefuglede}.

If the $b_j$ are collinear, $C$ itself is normal. Otherwise, partition the indices according to the lines through the origin containing $b_j$. There are at least two classes. For indices in different classes,
$b_i\bar b_j-\bar b_i b_j\ne0$ and $b_i\ne b_j$.
Thus \eqref{eq:powercommutator} and \eqref{eq:finitefuglede} force
$f_m(b_i)=f_m(b_j)$.
Any two indices in the same class can be joined through an index in another class, so all $f_m(b_j)$ coincide. The converse follows from part (1) and $C^m=Q^m$ when the variance vanishes. Finally, $f_m$ is monic of degree $m$, so a level set of $f_m$ contains at most $m$ distinct points.
\end{proof}

\begin{remark}[Polynomial functions]\label{rem:functional}
The same argument applies to an arbitrary polynomial $h(t)=\sum_{m=0}^d c_mt^m$. Put
$f_h=\sum_{m=1}^d c_mf_m$. Then
\begin{align*}
 \sum_j|h(\lambda_j)|^2&\le\|h(C)\|_F^2\le\|h(Q)\|_F^2,\\
 \|h(Q)\|_F^2-\|h(C)\|_F^2
 &=\sigma\sum_jr_j|f_h(b_j)|^2
       -\abs{\sum_jr_jf_h(b_j)}^2.
\end{align*}
The matrix $h(Q)$ is normal exactly when the $f_h(b_j)$ coincide; $h(C)$ is normal exactly when this holds or the $b_j$ are collinear. To include the constant term in the norm calculation, use
$L=c_0K^{-1}+\sum_{m\ge1}c_mL_m$ in Lemma~\ref{lem:symmetricnorm}. The vector $Lu$ differs from $f_h(D)u$ by a constant multiple of $u$, which does not change the variance. The commutator formula follows by linearity, and the finite-spectrum argument above applies with $h(X)$ in place of $X^m$.
\end{remark}

\subsection{The quartic bound and higher-order examples}

\begin{proof}[Proof of Corollary~\ref{cor:quartic}]
With $v=Du$, one has $L_2=D^2+vv^T$. Directly,
\[
 \|L_2\|_F^2=V_0+2V_1+S^2,\qquad
 \|L_2u\|^2=W,\qquad u^TL_2u=T+s^2.
\]
Lemma~\ref{lem:symmetricnorm} gives both right-hand sides of \eqref{eq:quartics}.

It remains to identify equality. If $f_2(b_j)$ is constant, the monic quadratic $f_2(t)=t^2+st$ takes that value at at most two distinct points. For two such points $a,c$,
\[
 0=\frac{f_2(a)-f_2(c)}{a-c}=a+c+s
   =(1+R_a)a+(1+R_c)c.
\]
Thus $a$ and $c$ lie on opposite rays. The alternative in part (2) of Theorem~\ref{thm:normalpowers} therefore adds no noncollinear configurations when $m=2$. Part (1) gives exactly \eqref{eq:twovalues} for equality in the unbalanced bound.
\end{proof}

At higher orders, noncollinear equality does occur.

\begin{example}[A noncollinear sixth-moment equality]\label{ex:cubic}
Take $n=2$, $r_1=r_2=1$, and
\[
 b_1=1,\qquad b_2=c=\frac{-5+i\sqrt{39}}8.
\]
Then $4c^2+5c+4=0$. For $Q=D(I+J)$, put $\tau=\tr Q=2(1+c)$. Its determinant is $3c=\tau^2$, so Cayley--Hamilton gives
\[
 Q^2-\tau Q+\tau^2I=0,\qquad Q^3=-\tau^3I.
\]
Consequently $C^3=Q^3$ is scalar, although $C$ is not normal. Hence equality holds throughout \eqref{eq:evensandwich} for $m=3$. In polynomial terms, this is the cubic $P(z)=z(z-1)(z-1/c)$.
\end{example}

\begin{example}[Regular polygons]\label{ex:polygon}
Let $r_j=\rho>0$ and $b_j=\beta\omega^{j-1}$, where $\beta\ne0$ and $\omega=e^{2\pi i/n}$. The determinant identity in Lemma~\ref{lem:representation} reduces to
\[
 \det(\lambda I-Q)=\lambda^n-(1+n\rho)\beta^n.
\]
Indeed, if $g(\lambda)=\lambda^n-\beta^n$, then
$\sum_jb_j\prod_{\ell\ne j}(\lambda-b_\ell)=\lambda g'(\lambda)-ng(\lambda)=n\beta^n$.
Thus
$Q^n=C^n=(1+n\rho)\beta^nI$.
For $n\ge3$ this gives noncollinear equality at order $2n$ and at every positive multiple of that order.
\end{example}

\section{Generating functions and stability}

\subsection{A closed generating determinant}

All the quadratic exterior refinements can be collected in one rank-two determinant.

\begin{proposition}\label{prop:generating}
For $t\ge0$, set
\[
 d_j(t)=1+t|b_j|^2,\qquad
 L(t)=\sum_j\frac{r_jt|b_j|^2}{d_j(t)},\qquad
 B(t)=\sum_j\frac{r_jb_j}{d_j(t)}.
\]
Then
\begin{equation}\label{eq:generating}
 \prod_j(1+t|\lambda_j|^2)
 \le \det(I+tC^*C)
 =\prod_jd_j(t)\bigl((1+L(t))^2+t|B(t)|^2\bigr).
\end{equation}
The inequality holds coefficient by coefficient as a polynomial in $t$. For any fixed $t>0$, equality holds if and only if the $b_j$ are collinear.
\end{proposition}

\begin{proof}
The coefficient of $t^k$ in $\det(I+tC^*C)$ is
$\|C^{[k]}\|_F^2=\B_k$.
Theorem~\ref{thm:exterior} proves the inequality and its equality statement (with the one-dimensional case automatic).

For the determinant formula, write $E=\diag(d_j(t))$, $v=Du$, and $s=\sum_jr_jb_j$. The matrix $C^*C$ is similar to $D^*KDK$, and
\[
 D^*KDK=|D|^2+|D|^2uu^T+\bar v(v^T+su^T).
\]
The rank-two determinant identity therefore reduces the quotient by $\det E$ to the determinant of
\[
 \begin{pmatrix}
 1+L & t\bar B\\
 s-B+sL & 1+L+ts\bar B
 \end{pmatrix}.
\]
Here $t\sum_jr_jb_j|b_j|^2/d_j=s-B$.
The displayed determinant is $(1+L)^2+t|B|^2$, as claimed.
\end{proof}

\subsection{A quantitative collinearity estimate}

Equality in the reciprocal Schoenberg inequality detects a line through the origin. Its deficit also controls the squared distance to the nearest such line. Keep $S,s,T$ from \eqref{eq:moments}, and define
\[
 F=\|C\|_F^2,\qquad
 \Delta=F-\sum_j|\lambda_j|^2,\qquad
 d_r^2=\min_{\theta\in\R}\sum_jr_j
              \bigl(\Ima(e^{-i\theta}b_j)\bigr)^2.
\]

\begin{proposition}\label{prop:stability}
Let
\[
 G=S^2-|T|^2+2\bigl(S|s|^2-\Rea(\bar s^{\,2}T)\bigr),
 \qquad
 \eta=\max_jr_j^{-1}+2+\sigma.
\]
Then
\begin{equation}\label{eq:stability}
 \Delta\ge\frac{G}{8\|C\|_{\rm op}^2}
       \ge\frac{G}{8F},\qquad
 \Delta\ge\frac{d_r^2}{4\eta}.
\end{equation}
Also,
\begin{equation}\label{eq:fourthdeficit}
 G=\tfrac12\|[C,C^*]\|_F^2
   =\tr\bigl((C^*C)^2\bigr)-\|C^2\|_F^2.
\end{equation}
No optimality is asserted for the constants in \eqref{eq:stability}.
\end{proposition}

\begin{proof}
Put $M=vv^*-\bar v v^T$, so $[C,C^*]=HMH$. Straightforward rank-one products give
\[
 \|M\|_F^2=2(S^2-|T|^2),\quad u^TMu=0,\quad
 \|Mu\|^2=2\bigl(S|s|^2-\Rea(\bar s^{\,2}T)\bigr).
\]
Expanding $\tr(MKMK)$ gives $\|[C,C^*]\|_F^2=2G$.
Expanding the square of the commutator proves the second identity in \eqref{eq:fourthdeficit}.

In a unitary Schur basis write $C=\Lambda+N$, where $\Lambda$ is diagonal and $N$ is strictly upper triangular. Then $\|N\|_F^2=\Delta$, and
\[
 [C,C^*]=[C,N^*]+[N,\Lambda^*].
\]
Since $\|\Lambda\|_{\rm op}\le\|C\|_{\rm op}$, it follows that
$\|[C,C^*]\|_F\le4\|C\|_{\rm op}\sqrt{\Delta}$.
This proves the first two estimates.

Finally, minimizing the elementary trigonometric expression defining $d_r^2$ gives
$d_r^2=(S-|T|)/2$. The second term in $G$ is nonnegative because $|T|\le S$, and hence
\[
 G\ge(S-|T|)(S+|T|)\ge2S d_r^2.
\]
Cauchy--Schwarz gives $|s|^2\le\sigma S$, so \eqref{eq:CF} implies $F\le\eta S$. Substitution proves the last estimate in \eqref{eq:stability}.
\end{proof}

\noindent \textbf{AI Declaration:} Artificial intelligence (AI) tools were used to test the inequalities in this paper and to provide examples. An LLM also provided feedback on an earlier draft, pointed out several extensions to its original results, and assisted in re-writing of this paper. The author assumes full responsibility for the accuracy of all results presented here.

\end{document}